\documentclass[oneside,openany,a4paper]{amsart}
\usepackage[a4paper, margin=0.67in]{geometry}
\usepackage{amssymb}
\usepackage[centertags]{amsmath}
\usepackage{amsthm}
\usepackage{amsfonts}
\usepackage{eucal}
\usepackage{mathrsfs}
\usepackage{enumerate}
\usepackage{multirow}
\usepackage{dsfont, soul, stackrel, tikz-cd, graphicx, etoolbox, multirow}
\usepackage{lipsum}

\usepackage[all,cmtip]{xy}
\usepackage{graphicx, eepic}

\usepackage{lmodern}
\usepackage[utf8]{inputenc}
\usepackage[T1]{fontenc}
\usepackage{pdfrender, xcolor}

\usepackage[utf8]{inputenc}
\usepackage[OT2,T1]{fontenc}
\DeclareSymbolFont{cyrletters}{OT2}{wncyr}{m}{n}
\DeclareMathSymbol{\Sha}{\mathalpha}{cyrletters}{"58}

\usepackage{pdfpages}
\usepackage{titletoc}
\usepackage{booktabs}
\usepackage[shortlabels]{enumitem}
\usepackage{mathtools}
\usepackage{thmtools}

\usepackage{epstopdf}
\usepackage{epsfig}
\usepackage{microtype}

\usepackage{hyperref}
\hypersetup{colorlinks}
\hypersetup{
    bookmarks=true,         
    unicode=false,          
    pdftoolbar=true,        
    pdfmenubar=true,        
    pdffitwindow=false,     
    pdfstartview={FitH},    
    pdftitle={My title},    
    pdfauthor={Author},     
    pdfsubject={Subject},   
    pdfcreator={Creator},   
    pdfproducer={Producer}, 
    pdfkeywords={keyword1} {key2} {key3}, 
    pdfnewwindow=true,      
    colorlinks=true,       
    linkcolor=red,          
    citecolor=magenta,        
    filecolor=violet,      
    urlcolor=blue      
}
\usepackage{multirow, url}
\usepackage{textcomp}
\DeclareMathAlphabet{\cmcal}{OMS}{cmsy}{m}{n}
\makeatletter
\def\blfootnote{\xdef\@thefnmark{}\@footnotetext}
\makeatother

\newtheoremstyle{thm}
  {3pt}
  {3pt}
  {\em}
  {0pt}
  {\bfseries}
  {}
  {5pt}
  {}
\newtheoremstyle{rem}
  {3pt}
  {3pt}
  {}
  {0pt}
  {\bfseries}
  {.}
  {5pt}
  {}

\newtheorem{thm}{Theorem}[section]

\newtheorem{lem}[thm]{Lemma}

\newtheorem{conj}[thm]{Conjecture}

\theoremstyle{definition}

\theoremstyle{rem}

\newtheorem{theorem}{Theorem}[section]
\newtheorem{corollary}[thm]{Corollary}
\newtheorem{lemma}[thm]{Lemma}
\newtheorem{proposition}[thm]{Proposition}
\newtheorem{definition}[thm]{Definition}
\newtheorem{remark}[thm]{{Remark}}

\numberwithin{equation}{section} \numberwithin{table}{section}
	
\newtheorem*{thm*}{Theorem}
\newtheorem*{rem*}{Remark}
\newtheorem*{rems*}{Remarks}
\newtheorem*{exam*}{Example}
\newtheorem*{exams*}{Examples}

\makeatletter
\newcommand{\neutralize}[1]{\expandafter\let\csname c@#1\endcsname\count@}
\makeatother

\usepackage{enumerate}

\def\bos#1{{\mathbf{#1}}}

  \newcommand{\nc}{\newcommand}
  \newcommand{\be}{\begin{eqnarray*}}
  \newcommand{\ee}{\end{eqnarray*}}
  \newcommand{\bea}{\begin{eqnarray}}
  \newcommand{\eea}{\end{eqnarray}}

   \nc{\bei}{\begin{itemize}}
   \nc{\eei}{\end{itemize}}
   \nc{\bee}{\begin{enumerate}}
   \nc{\eee}{\end{enumerate}}
   \nc{\bet}{\begin{theorem}}
   \nc{\eet}{\end{theorem}}
   \nc{\bed}{\begin{definition}}
   \nc{\eed}{\end{definition}}
   \nc{\bel}{\begin{lemma}}
   \nc{\eel}{\end{lemma}}
   \nc{\bep}{\begin{proposition}}
   \nc{\eep}{\end{proposition}}
   \nc{\bec}{\begin{corollary}}
   \nc{\eec}{\end{corollary}}
   \nc{\ber}{\begin{remark}}
   \nc{\eer}{\end{remark}}
   \nc{\beex}{\begin{example}}
   \nc{\eeex}{\end{example}}

   \nc{\bpm}{\begin{pmatrix}}
   \nc{\epm}{\end{pmatrix}}
   \nc{\bspm}{\left(\begin{smallmatrix}}
   \nc{\espm}{\end{smallmatrix}\right)}

\newcommand{\cA}{\mathcal{A}}

\newcommand{\cF}{\mathcal{F}}

\newcommand{\cH}{\mathcal{H}}

\newcommand{\cO}{\mathcal{O}}

\newcommand{\bC}{\mathbb{C}}

\newcommand{\bF}{\mathbb{F}}

\newcommand{\bQ}{\mathbb{Q}}

\newcommand{\bZ}{\mathbb{Z}}

\nc{\frf}{\mathfrak{f}} 

\newcommand{\frp}{\mathfrak{p}}      
   
\nc{\frs}{\mathfrak{s}}  
\nc{\frt}{\mathfrak{t}} 
\nc{\fru}{\mathfrak{u}}
  
\nc{\lsl}{\mathfrak{sl}}
\nc{\lgl}{\mathfrak{gl}}

\nc{\upsi}{\underline{\psi}}
\nc{\uchi}{\underline{\chi}}

\DeclareMathOperator{\Tr}{Tr}
\DeclareMathOperator{\Gal}{Gal}
\DeclareMathOperator{\Hom}{Hom}

\DeclareMathOperator{\Spec}{Spec}

\DeclareMathOperator{\End}{End}

\newcommand{\lra}{\longrightarrow}    

\nc{\surjto}{\twoheadrightarrow}
\nc{\ts}{\times}
\nc{\ds}{\displaystyle}
\nc{\nd}{\noindent}  
\nc{\ud}{\underline}
\nc{\ov}{\overline}
\nc{\maplra}[1]{\buildrel #1 \over \lra}
\nc{\mapto}[1]{\buildrel #1 \over \to}
\nc{\setb}[1]{\{  #1\}}

 \nc{\cHom}{\mathcal{H}om}

\nc{\cdruur}[8] {\begin{CD} 
#1 @>#2>> #3\\ 
@AA#4A @AA#5A\\ 
#6 @>#7>> #8 
\end{CD} }
\nc{\cdrddr}[8] {\begin{CD} 
#1 @>#2>> #3\\ 
@VV#4V @VV#5V\\ 
#6 @>#7>> #8 
\end{CD} }

\nc{\dia}[8]{\xymatrix{ 
&#1 \ar@{-}[ld]_{#2} \ar@{-}[rd]^{#3} \\
#4 \ar@{-}[rd]_{#6} & &#5 \ar@{-}[ld]^{#7}\\ 
&#8} }

\nc{\diam}[9]{\xymatrix{ 
&#1 \ar@{-}[ld]_{#2}  \ar@{-}[d]^{#3} \ar@{-}[rd]^{#4} \\
#5 \ar@{-}[rd]_{#8}     & #6 \ar@{-}[d]_{#9}      & #7   \ar@{-}[ld]^{2} \\
& \bQ} } 

\nc{\sumn}[2][n]{#2_{1} +#2_{2}+ \cdots + #2_{#1}}
\nc{\poly}[3][n]{#2_{#1}#3^{#1} +#2_{#1-1}#3^{#1-1}  \cdots + #2_{1} #3+ #2_0}
\nc{\dpoly}[3][n]{#1#2_{#1}#3^{#1-1} +(#1-1)#2_{#1-1}#3^{#1-1}  \cdots +2 #2_{2} #3+ #2_1}
\nc{\mpoly}[3][n]{#3^{#1} +#2_{#1-1}#3^{#1-1}  \cdots + #2_{1} #3+ #2_0}

\nc{\vpar}[4]{    \left \{ \begin{array}{cc} #1 & \textrm{if } #2, \\
&\\
#3 & \textrm{if } #4. 
\end{array}\right. }

\nc{\vparr}[4]{    \left \{ \begin{array}{cc} #1 & \textrm{if } #2, \\
&\\
#3 & \textrm{if } #4, 
\end{array}\right. }

\nc{\ary}[5]{#1: \left\{ \begin{array}{ll} #2 &\mapsto #3 \\ #4 &\mapsto #5 \end{array} \right.}
 \nc{\bedm}{\begin{displaymath}}
 \nc{\eedm}{\end{displaymath}}
 \nc{\art}{\hbox{\bf Art}^\Z}
 \nc{\bvx}{\bos{B\!\!V}_{\! \!X}}

\newcommand{\pmat}{\left(\begin{matrix}}   
\newcommand{\epmat}{\end{matrix}\right)}   
\newcommand{\psmat}{\left(\begin{smallmatrix}}    
\newcommand{\epsmat}{\end{smallmatrix}\right)}
\nc{\twotwo}[4]{\pmat #1 & #2 \\ #3 & #4 \epmat}
\nc{\thrthr}[9]{\pmat #1 & #2 & #3 \\ #4 & #5 & #6 \\ #7 & #8 & #9 \epmat}
\nc{\stwotwo}[4]{\psmat #1 & #2 \\ #3 & #4 \epsmat}
\nc{\sthrthr}[9]{\psmat #1 & #2 & #3 \\ #4 & #5 & #6 \\ #7 & #8 & #9 \epsmat}

\def\eqalign#1{\null\,\vcenter{\openup\jot\m@th
\ialign{\strut\hfil$\displaystyle{##}$&$\displaystyle{{}##}$\hfil
\crcr#1\crcr}}\,}

\def\eqn#1#2{
\xdef #1{(\nsecsym\the\meqno)}
\global\advance\meqno by1
$$#2\eqno#1\eqlabeL#1
$$}

\def\a{\alpha}

\def\g{\gamma}

\def\cA{{\mathcal A}}

\def\R{\mathbb{R}}
\def\C{\mathbb{C}}

\def\Z{\mathbb{Z}}

\def\Q{\mathbb{Q}}

\def\ms{\mathfrak{s}}

\def\rd{\partial}

\def\Ker{\hbox{Ker}\;}

\def\mod{\hbox{ }mod\hbox{ }}

\begin{document}


\catcode`\@=11 

\global\newcount\nsecno \global\nsecno=0
\global\newcount\meqno \global\meqno=1
\def\newsec#1{\global\advance\nsecno by1
\eqnres@t
\section{#1}}
\def\eqnres@t{\xdef\nsecsym{\the\nsecno.}\global\meqno=1}
\def\sequentialequations{\def\eqnres@t{\bigbreak}}\xdef\nsecsym{}

\def\draftmode{\message{ DRAFTMODE }

{\count255=\time\divide\count255 by 60 \xdef\hourmin{\number\count255}
\multiply\count255 by-60\advance\count255 by\time
\xdef\hourmin{\hourmin:\ifnum\count255<10 0\fi\the\count255}}}
\def\nolabels{\def\wrlabeL##1{}\def\eqlabeL##1{}\def\reflabeL##1{}}
\def\writelabels{\def\wrlabeL##1{\leavevmode\vadjust{\rlap{\smash%
{\line{{\escapechar=` \hfill\rlap{\tt\hskip.03in\string##1}}}}}}}%
\def\eqlabeL##1{{\escapechar-1\rlap{\tt\hskip.05in\string##1}}}%
\def\reflabeL##1{\noexpand\llap{\noexpand\sevenrm\string\string\string##1}
}}

\nolabels

\def\eqn#1#2{
\xdef #1{(\nsecsym\the\meqno)}
\global\advance\meqno by1
$$#2\eqno#1\eqlabeL#1
$$}

\def\eqalign#1{\null\,\vcenter{\openup\jot\m@th
\ialign{\strut\hfil$\displaystyle{##}$&$\displaystyle{{}##}$\hfil
\crcr#1\crcr}}\,}

\def\ket#1{\left|\bos{ #1}\right>}\vspace{.2in}
   \def\bra#1{\left<\bos{ #1}\right|}
\def\oket#1{\left.\bos{ #1}\right>}
\def\obra#1{\left<\bos{ #1}\right.}
\def\epv#1#2#3{\left<\bos{#1}\left|\bos{#2}\right|\bos{#3}\right>}
\def\qbvk#1#2{\bos{\left(\bos{#1},\bos{#2}\right)}}
\def\Hoch{{\tt Hoch}}
\def\rrd{\up{\rightarrow}{\rd}}
\def\lrd{\up{\leftarrow}{\rd}}
   \nc{\hr}{[\![\hbar]\!]}
   \nc{ \cAb}{\cA\!(b)}
   \nc{\bNn}{\bZ_{\geq 0}}
   \nc{\Ab}{A\!(b)}
   \nc{\modulo}{\operatorname{mod}}
   
\def\foot#1{\footnote{#1}}

\catcode`\@=12 

\def\fr#1#2{{\textstyle{#1\over#2}}}
\def\Fr#1#2{{#1\over#2}}
\def\ato#1{{\buildrel #1\over\longrightarrow}}

\newcommand{\QFT}{\operatorname{QFT}}
\newcommand{\zell}{{\Z/{\ell\Z}}}
\newcommand{\muell}{{\mu_{\ell}}}
\newcommand{\Gm}{{\mathbb{G}}_{\operatorname{m}}}
\newcommand{\Zmod}[1]{\overline{\Z/{#1}\Z}}
\newcommand{\zmod}[1]{{\Z/{#1}\Z}}
\newcommand{\zmodd}[1]{{{\frac{1}{#1}}\Z}/\Z}
\newcommand{\Inv}{{\operatorname{Inv}}}
\newcommand{\inv}{{\operatorname{inv}}}
\newcommand{\Et}{{\operatorname{Et}}}
\def\invlim{\varprojlim}
\def\nZ{\frac{1}{n}\Z/\Z}
\newcommand{\pair}[2]{\langle #1, \,#2 \rangle}
\newcommand{\legendre}[2]{\genfrac{(}{)}{}{}{#1}{#2}}
\newcommand{\xyv}[1]{\xymatrixrowsep{#1 pc}}
\def\ms{\medskip}
\newcommand{\dlog}{\mathrm{dlog}}
\newcommand{\on}{\operatorname}
\def\lk{\ell k}
\newcommand\rat{{\mathbb Q}}
\newcommand\ratinfp{{\rat_{\infty,p}}}
\newcommand{\Ext}{{\operatorname{Ext}}}
\newcommand{\gl}{\operatorname{\mathfrak{l}}}
\newcommand{\gp}{\operatorname{\mathfrak{p}}}
\newcommand{\Tei}{\operatorname{Tei}}
\newcommand{\barQp}{\operatorname{\overline{\mathbf{Q}}_{\it{p}}}}
\draftmode

\newcommand{\loc}{\operatorname{loc}}

\newcommand{\bd}{\boldsymbol}
\newcommand{\Ent}{\operatorname{Ent}}

\title{Entanglement entropies in the abelian arithmetic Chern-Simons theory}

\author{Hee-Joong Chung}
\email{hjchung@jejunu.ac.kr}
\address{H.-J. C.: Department of Science Education, Jeju National University, 102 Jejudaehak-ro, Jeju-si, Jeju-do, 63243, South Korea}

\author{Dohyeong Kim}
\email{dohyeongkim@snu.ac.kr}
\address{D.K: Department of Mathematical Sciences and Research Institute of Mathematics, Seoul National University, 1 Gwanak-ro, Gwanak-gu, Seoul 08826, South Korea}

\author{Minhyong Kim}
\email{minhyong.kim@icms.org.uk}
\address{M.K.: International Centre for Mathematical Sciences, 47 Potterrow,  Edinburgh EH8 9BT\\
Korea Institute for Advanced Study, 85 Hoegiro, Dongdaemungu, Seoul, South Korea}

\author{Jeehoon Park}
\email{jpark.math@gmail.com}
\address{J.P: QSMS, Seoul National University, 1 Gwanak-ro, Gwanak-gu, Seoul 08826, South Korea}

\author{Hwajong Yoo}
\email{hwajong@snu.ac.kr}
\address{H.Y.: College of Liberal Studies and Research Institute of Mathematics,
Seoul National University, 1 Gwanak-ro, Gwanak-gu,  Seoul 08826, South Korea
}

\date{}


\maketitle

\begin{abstract}
The notion of {\em entanglement entropy} in quantum mechanical systems is an important quantity, which measures how much a physical state  is entangled in a composite system. Mathematically, it measures how much the state vector is not decomposable as elements in the tensor product of two Hilbert spaces. In this paper, we seek its arithmetic avatar: the theory of arithmetic Chern-Simons theory with finite gauge group $G$ naturally associates a state vector inside the product of two quantum Hilbert spaces and we provide a formula for the {\em von Neumann entanglement entropy} of such state vector when $G$ is a cyclic group of prime order.

\end{abstract}

\tableofcontents

\section{Introduction}
{\em Arithmetic topological quantum field theory} is an area that is gradually growing into a substantial direction of research based on examples \cite{kim, CKKPY20,  CKKPPY19, HKM} and analogies \cite{CCKKPY} as well as the potential to cast light on phenomena as deep as the Langlands programme \cite{BSV}. This paper is mostly concerned with a new arithmetic invariant of collections of primes that imports to number theory a well-known quantity  in quantum mechanics that has received a good deal of attention in recent years in the context of quantum foundation and quantum computation \cite{bell}.

The notion of {\em entanglement} in quantum mechanical systems is  straightforward from a mathematical point of view. The composite of two systems $A$ and $B$ has state space modelled by a tensor-product Hilbert space
$$\cH=\cH_A\otimes \cH_B$$
and an entangled state in $\cH$ is merely a vector that is not decomposable. Recall that a decomposable vector in $\cH$ is one of the form $v\otimes w$. In particular, in any reasonable sense,  a state is entangled with probability 1. This elementary notion becomes of interest in physics already in the simplest situation when we have a state of the form
$$\psi:=\frac{1}{\sqrt{2}}(v_1\otimes w_1+v_2\otimes w_2),$$
where $v_1, v_2$ are orthonormal eigenvectors for an observable $F$ of system $A$ and $w_1, w_2$ are orthonormal eigenvectors for an observable $G$ of system $B$.
In this case, $v_1\otimes w_1$ and $v_2\otimes w_2$ are orthonormal eigenstates of the observables $F\otimes I, I\otimes G,$  and $F\otimes G$ of the composite system and $\psi$ is a superposition of them. Thus, if a measurement is made of $F\otimes I$, for example, then $\psi $ will `collapse' into $v_1\otimes w_1$ or $v_2\otimes w_2$. Which is the case can be determined just by looking at system $A$.  However, if the state of system $A$ is $v_1$, then the state of system $B$ must be $w_1$, in spite of the fact that $G$ has not been measured at all. Thus, the states of the two systems have become `entangled'. Particular interest is attached to the situation where $A$ and $B$ denote different regions of space.  The state $\psi$ as above indicates that the state of the system in a region could be entangled with the state in a very distant region.

\subsection{TQFT and entanglement entropy}
In \cite{BFLP}, a variation on this scenario is considered where $M$ is an $n$-manifold with boundary $\partial M= A\sqcup  B$, and we are given an $n$-dimensional unitary TQFT (topological quantum field theory) $Z$. Then the theory will assign a vector in the tensor product of  Hilbert spaces $Z_A$ and $Z_B$
\be 
Z_M\in Z_A\otimes Z_B,
\ee
that can be entangled. One interpretation is that $A\sqcup B$ is space, and the theory has created the state $Z_M$ from the vacuum along the spacetime $M$. In particular, the authors in \cite{BFLP} take $M$ to be the complement of  tubular neighbourhoods of two knots $C,D$, in which case the entanglement of $Z_M$ can be considered as a quantum manifestation of linking, in some sense.

It is useful at this point to utilize a well-known numerical measure of entanglement, namely, the {\em von Neumann entanglement entropy}. The definition is made in terms of partial traces and the formalism of mixed states.
For this, we will avoid technicalities by assuming all Hilbert spaces (Hermitian inner product spaces) to be finite-dimensional. One takes a normalised state $\psi\in \cH=\cH_A\otimes \cH_B$ and regards it as a projection operator $\pi_{\psi}: \cH\to \cH$. We have the isomorphism
$$\End(\cH_A\otimes \cH_B)\simeq \End(\cH_A)\otimes \End(\cH_B)$$
and the linear map
$$\Tr: \End(\cH_B)\to \C.$$
Hence, there is a `partial trace' map
$$\Tr_B: \End(\cH_A\otimes \cH_B)\to \End(\cH_A).$$
We define the reduced density matrix of $\psi$ to be 
$$\rho_{\psi, A}:=\Tr_{B}(\pi_\psi):\cH_A \to \cH_A,$$
using which the entanglement entropy of $\psi$ is defined by 
\bea\label{entropy}
\Ent(\psi):=-\Tr(\rho_{\psi, A}\log \rho_{\psi, A}).
\eea
Even though such an expression for entropy may be familiar, the computation is not entirely easy. 
It can be shown from the singular value decomposition that the expression
$$-\Tr(\rho_{\psi, B}\log \rho_{\psi, B}) $$
yields the same number.
For the moment, we merely note that $\Ent(\psi)=0$ if and only if $\psi$ is decomposable.

\subsection{Arithmetic TQFT and the main theorem}
The upshot of \cite{BFLP} is that $\Ent(Z_M)$ is a refined linking invariant when $Z$ is a topological quantum field theory like Chern-Simons theory. It is this framework one can try to recreate in the setting of arithmetic topological quantum field theory. 
We fix a prime $p$.
Let $F$ be a totally imaginary number field and $S$ be a finite set of primes in the ring $\cO_F$ of integers such that $S$ contains all the prime ideals dividing $p$.
Let 
$${X_S}=\Spec(\cO_F)\setminus S.$$
A choice of a finite gauge group $G$ and a 3-cocycle in $H^3(G,\Z/p\Z)$ determines the arithmetic Chern-Simons theory. Then the theories of \cite{kim, CKKPY20, HKM} associate a normalised state vector (see \eqref{vector} for details)
\bea\label{invector}
Z_{X_{S_1,S_2}} \in \cH_{S_1} \otimes \cH_{S_2}
\eea
 in the product of two Hilbert spaces $\cH_{S_1}$ and $\cH_{S_2}$, whenever we have a decomposition $S=S_1\sqcup S_2$. 
The main goal of this paper will be to compute the entanglement entropy of $Z_{X_{S_1,S_2}}$ when $G$ is cyclic of prime order $p$, thereby obtaining a sense of the  information it contains. 

Let $\Pi^S=\Gal({F}^S/{F})$ where ${F}^S$ is the maximal unramified extension of ${F}$ outside $S$.
We consider the case $G=\Z/p\Z$; let
$$
\cF_{{X_S}}:=\Hom_{cts}(\Pi^S, \Z/p\Z)=H^1(\Pi^S,\Z/p\Z)
$$
be the set of continuous group homomorphisms from $\Pi^S$ to $\Z/p\Z$.
Consider the localisation maps 
$$\loc^{S}_{S_i}: H^1(\Pi^S, \Z/p\Z)\to \prod_{\frp \in S_i} H^1(\Pi_\frp, \Z/p\Z), \quad i=1,2$$
where $\Pi_\frp=\Gal(\overline{F}_{\frp}/F_{\frp})$ is the absolute Galois group of $F_\frp$. 
A rather simple formula is given by the following:

\begin{theorem}\label{gthm}
The entanglement entropy of $Z_{X_{S_1,S_2}}$ associated to the arithmetic Chern-Simons theory for the gauge group $G=\Z/p\Z$ and the 3-cocycle\footnote{See \eqref{cchoice} for its precise definition.} $[c]\in H^3(\Z/p\Z, \Z/p\Z)$, is given by
$$
\Ent(Z_{X_{S_1,S_2}}) =\left(\dim_{\bF_p}(\cF_{{X_S}})-\dim_{\bF_p} \biggr(\Ker(\loc^S_{S_1}) + \Ker(\loc^S_{S_2})\biggl)\right)\log p.
$$
\end{theorem}
Even though not much structure is evident in the formula, it does bring forth the interaction of invariants that are normally somewhat difficult to consider, namely, the images of the two separate localisation maps from $S$-ramified cohomology. The image of the localisation to both sets of primes is the subject of the Poitou-Tate duality, but each of the individual images is more mysterious and not subject to much general analysis. 
The Poitou-Tate duality tells us about 
$$\dim(\Ker(\loc^{S}_{S_1})\cap \Ker(\loc^{S}_{S_1})).$$
When combined with above theorem, we can compute
$$\dim(\Ker(\loc^{S}_{S_1})+\Ker(\loc^{S}_{S_1})),$$
a quantity not normally considered in classical duality theory. 
The TQFT analogy calls attention thereby to this  natural arithmetic quantity. 

It should be admitted that the results of this paper do not yet make clear that the notion of entanglement is useful for number theory.\footnote{Though $\Ent(Z_{X_{S_1,S_2}})$ was expected to contain some sort of information on arithmetic linking numbers of primes like the topological case \cite{BFLP}, it turns out not in our setting. One might try to change $[c]$ and $G$ to obtain such information, but one seems to need a new idea for an explicit computation of entanglement entropy.} In future work, we will investigate entanglement for different choices of a 3-cocycle $[c]$ and a gauge group $G$. Moreover we will study entanglement for arithmetic BF theories, relate the computation of entanglement entropy to the path integrals of `L-function type' in \cite{CK}, and seek out applications to questions of concrete arithmetic interest. In the meanwhile, it is hoped that the existence of new invariants of number fields and primes that make use of essentially classical machinery within a TQFT framework will be of intrinsic interest to arithmeticians and mathematical physicists.

\subsection{Acknowledgement}

The work of D.K. was supported by the National Research Foundation of Korea (2020R1C1C1A0100681913) and by Samsung Science and Technology Foundation (SSTF-BA2001-01).
M.K. was supported in part by  UKRI grant EP/V046888/1 and a Simons Fellowship at the Isaac Newton Institute. The work of J.P. was supported by the National Research Foundation of Korea (NRF-2021R1A2C1006696) 
and the National Research Foundation of Korea (NRF) grant funded by the Korea government (MSIT) (No.2020R1A5A1016126). The work of H.Y. was supported by the National Research Foundation of Korea(NRF) grant funded by the Korea government(MSIT) (No. 2020R1A5A1016126).

\section{Abelian arithmetic Chern-Simons theory and entanglement entropy} \label{sec2.1}

We fix a gauge group $G=\Z/p\Z$.
Assume that $F$ contains $\mu_{p^2}$
(where $\mu_n$ is the group of $n$-th roots of unity in the algebraic closure $\overline F$).
Consider the Bockstein exact sequence
$$
0 \xrightarrow{} \bZ/p\bZ \xrightarrow{}  \bZ/p^2\bZ  \xrightarrow{\mod p}  \bZ/p\bZ.
$$
Let $s$ be a set-theoretic splitting of $\mod p$-map.
Then this gives us a 3-cocycle
\begin{eqnarray}\label{cchoice}
c=\a \cup d (s (\a)) \in Z^3(\Z/p\Z, \Z/p\Z)
\end{eqnarray}
where 
$$
(\a: \Z/p\Z \to \Z/p\Z) \in H^1(\Z/p\Z, \Z/p\Z)
$$
is the identity map and $d$ is the differential on inhomogeneous cochains.
From now on, we fix a section $s$ and the 3-cocycle $c$, and work with
the ACST (arithmetic Chern-Simons theory) as developed in \cite{kim, CKKPY20, HKM} associated to $F, G,$ and $c$.
Since $G$ is abelian, we call such theory the abelian ACST.

\subsection{Hilbert spaces and state vectors}
We fix a finite set $S$ of prime ideals of $\cO_{{F}}$, which contains all the prime ideals dividing $p$. 
Let $r$ be the cardinality of $S$ and $S=\{\frp_1, \ldots, \frp_r\}$.
Here we briefly review how to construct the finite dimensional Hilbert space $\cH_S$ and a normalised vector in $\cH_S$ in the abelian ACST.

For any prime ideal $\frp$ of $\cO_{F}$, let $\Pi_{\frp}=\Gal(\overline {F}_{\frp}/{F}_\frp)$ where $\overline {F}_{\frp}$ is the algebraic closure of the local field ${F}_\frp$. We define 
\bea \label{localfields}
\cF_S=\cF_{Y_S}:=\prod_{i=1}^r \cF_{Y_{\frp_i}}, \quad \cF_\frp=\cF_{Y_\frp}:=\Hom_{cts}(\Pi_{\frp}, \Z/p\Z).
\eea
The general theory \cite{kim, CKKPY20, HKM} associates a $\Z/p\Z$-torsor $\mathcal{CS}_S$ over $\cF_S$
(Definition \ref{torsordef}).
Let us briefly recall it here.
For $\frp \in S$, we define
$$
\mathcal{CS}_{\frp}(\rho_\frp):=d^{-1}(c \circ \rho_\frp) \mod B^2(\Pi_\frp, \Z/N\Z), \quad \rho_\frp \in \cF_{\frp}.
$$
Then $\mathcal{CS}_\frp(\rho_\frp)$ becomes a $\Z/p\Z$-torsor via the local invariant map $\inv_\frp: H^2(\Pi_{\frp}, \Z/p\Z) \simeq \Z/p\Z$.
Then we define
$$
\varpi_{\frp}:\mathcal{CS}_\frp=\bigsqcup_{\rho_\frp \in \cF_\frp} \mathcal{CS}_\frp(\rho_\frp) \to \cF_\frp, \quad \varpi_{\frp}(\alpha_\frp):=\rho_\frp
$$
where $\alpha_\frp \in \mathcal{CS}_\frp(\cF_\frp)$ (i.e., $d(\alpha_\frp)=c\circ \rho_\frp$). This $\mathcal{CS}_\frp$ becomes a $\Z/p\Z$-torsor over $\cF_\frp$.
For $\rho_S=(\rho_{\frp_1}, \ldots, \rho_{\frp_r}) \in \cF_S$, we define
$$
\mathcal{CS}_S(\rho_S):=\prod_{i=1}^r \mathcal{CS}_{\frp_i}(\rho_{\frp_i})/\sim
$$
where $(\alpha_1, \ldots, \alpha_r) \sim (\alpha_1', \ldots, \alpha_r')$ if and only if $\sum_{i=1}^r \inv_{\frp_i} (\alpha_i - \alpha_i')=0$.
Let 
$
\mathcal{CS}_S=\bigsqcup_{\rho_{S} \in \cF_{S}} \mathcal{CS}_{S}(\rho_{S}).
$
\begin{definition}\label{torsordef}
Define a $\Z/p\Z$-torsor $\varpi_S : \mathcal{CS}_S \to \cF_S$ by
$$
\varpi_S(\alpha_S)=(\varpi_{\frp_1}(\alpha_{\frp_1}), \ldots, \varpi_{\frp_1}(\alpha_{\frp_1})), \quad \alpha_S = [(\alpha_{\frp_1}, \ldots, \alpha_{\frp_r})]\in \mathcal{CS}_S(\rho_S).
$$
\end{definition}

We associate a $\C$-line bundle $\pi_S: {CS}_S\to \cF_S$ to $\varpi_{S}: \mathcal{CS}_{S} \to \cF_{S}$
$$
{CS}_S=\mathcal{CS}_S \times_{\Z/p\Z} \C=\mathcal{CS}_S \times \C/\sim
$$
where $([\alpha_S], z) \sim ([\alpha_S] \cdot m, \zeta_p^{-m} z)$ for $m \in \Z/p\Z$ and $\zeta_p$ is a primitive $p$-th root of unity.

\begin{definition}[Hilbert spaces]
We define
$$
\cH_S:=\Gamma(\cF_S, {CS}_S) 
$$
to be the space of global sections of the line bundle ${CS}_S$. Note that $\cH_S$ has a canonical Hilbert space structure because the line bundle ${CS}_S$ comes from an $U(1)$-torsor (and hence has a Hermitian metric). Let $||\cdot||$ be the associated norm on $\cH_S$.
\end{definition}

The natural embedding $\iota_{\frp_i}: \Pi_{\frp_i} \to \Pi^S$ gives us the restriction homomorphism
\be
\loc_S: \cF_{{X_S}} \to \cF_S, \quad \rho\mapsto (\rho \circ \iota_{\frp_i})_i.
\ee
Note that there is a map  $\iota: \mathcal{CS}_S(\rho_S) \to {CS}_S(\rho_S)$ defined by 
$
\iota ([\alpha]):= [([\alpha], 1)].
$
For each $\rho\in \cF_{{X_S}}$, we define
\bea\label{qinv}
CS_{{X_S}}(\rho):=\iota(\mathcal{CS}_{{X_S}}(\rho))= \iota [\loc_S( \beta_\rho)] \in \pi_S^{-1}(\rho_S), \quad \rho_S = \loc_S( \rho)
\eea
where $\beta_\rho \in C^2(\Pi^S, \Z/p\Z) \mod B^2(\Pi^S, \Z/p\Z)$ is given by\footnote{By using global class field theory, one can show that $\mathcal{CS}_{{X_S}}(\rho)$ does not depend on a choice of $\beta_\rho$.
}
\begin{eqnarray}\label{betarho}
d(\beta_\rho) = c \circ \rho
\end{eqnarray}
due to the fact $H^3(\Pi^S, \Z/p\Z)=0$ (since $S$ contains all the prime ideals dividing $p$).

\begin{definition}[state vectors]
Let
\bea \label{fromglobal}
 \cF_{{X_S}}(\rho_S):=\{\rho \in \cF_{{X_S}}: \loc_S (\rho) =\rho_S\}, \quad \rho_S\in \cF_S.
\eea
We define the normalised quantum arithmetic Chern-Simons invariant $Z_{{X_S}} =\frac{\tilde Z_{{X_S}}}{||\tilde Z_{{X_S}}||}\in \cH_S$ with boundary $S$ as follows:
\bea\label{nvector}
\tilde Z_{{X_S}} (\rho_S) := \frac{1}{p}\sum_{\rho \in \cF_{{X_S}}(\rho_S)}{CS}_{{X_S}} (\rho)\in \pi_S^{-1}(\rho_S), \quad \rho_S \in \cF_S. 
\eea
We call $Z_{X_S}$ a normalised state vector in $\cH_S$.
\end{definition}

\subsection{The bipartite entanglement entropy and set-up of the problem}

We take a partition $S=S_1 \sqcup S_2$ where neither of $S_1$ nor $S_2$ is empty.
For a Hilbert space $\cH$, we denote by $\cH^1$ the set of norm 1 vectors. 
The bipartite entanglement entropy 
$$
\Ent:(\cH_{S_1} \otimes \cH_{S_2})^1 \to \R
$$ 
was defined in \eqref{entropy} as the von Neumann entropy of either of its reduced states; the result is independent of which one we pick, since they are of the same value (can be proved from the Schmidt decomposition of the state with respect to the bipartition). 
%

For a partition $S=S_1 \sqcup S_2$ above, the natural identification $\cF_S = \cF_{S_1} \times \cF_{S_2}$ induces a canonical isomorphism of Hilbert spaces
$$
\Theta_S^{S_1,S_2}: \cH_S \xrightarrow{\sim} \cH_{S_1} \otimes \cH_{S_2}.
$$
Using this isomorphism, we define the normalised vector in the Hilbert space $\cH_{S_1}\otimes \cH_{S_2}$,
\bea \label{vector}
Z_{X_{S_1,S_2}} = \Theta_S^{S_1,S_2} (Z_{{X_S}})\in \cH_{S_2} \otimes \cH_{S_2}, 
\eea
which was mentioned in \eqref{invector} of the introduction, whose entanglement entropy is our main interest.

\begin{remark}
In our finite arithmetic Chern-Simons theory, one may consider a more ``physical'' quantum abelian arithmetic Chern-Simons invariant using the Lagrangian $\loc_S(\cF_{X_S})$ of $\cF_S$ (Lemma \ref{lgr}): we may view $Z_{{X_S}}$ as an element of 
$\cH_S^{glob}:=\Gamma (\loc_S(\cF_{X_S}), {CS}_S)$
where 
\be
\tilde Z_{{X_S}} (\rho_S) :=\frac{1}{p} \sum_{\rho \in \cF_{{X_S}}(\rho_S)} CS_{{X_S}} (\rho).
\ee
Since $\loc_S(\cF_{X_S}) = \loc_{S_1}(\cF_{X_{S_1}}) \times \loc_{S_2}(\cF_{X_{S_2}})$, we have an induced identification
\be
\cH_S^{glob} \simeq \cH_{S_1}^{glob} \otimes \cH_{S_2}^{glob}
\ee
with respect to which one can examine entanglement entropy.
In this case we find the entanglement entropy always zero due to Lemmata \ref{trivial} and \ref{triviality}.
Thus we concentrate on the situation $Z_{X_{S_1,S_2}} \in \cH_{S_1} \otimes \cH_{S_2}$,
 where one regards $\cF_{S}$ as a discrete space as in \cite{FQ}.
\end{remark}

\subsection{A computational tool: an explicit trivialisation of the line bundle}

%

In order to compute entanglement entropies, we choose a section $x_S$ to $\varpi_S:\mathcal{CS}_S \to \cF_S$, i.e., a map $x_S:\cF_S \to \mathcal{CS}_S$ such that $\varpi_S \circ x_S=id$, which enables us to construct an explicit trivialisation of $\varpi_S:\mathcal{CS}_S \to \cF_S$: we consider 
$$
\mathcal{CS}_S^{x_S}:=\cF_S \times \Z/p\Z, \quad \varpi_S^{x_S}=pr_1:\mathcal{CS}_S^{x_S} \to \cF_S
$$
where $pr_1: \mathcal{CS}_S^{x_S} \to \cF_S$ is the projection to the first factor.
Then $\mathcal{CS}_S \simeq \mathcal{CS}_S^{x_S}$ are isomorphic as $\Z/p\Z$-torsors.
Under this isomorphism, \eqref{qinv} can be interpreted as follows:
\bea \label{pboundary}
\mathcal{CS}_{{X_S}}^{x_S} (\rho)=\sum_{i=1}^r \inv_{\frp_i} \big( \loc^S_{\frp_i}( \beta_\rho) - x_{\frp_i}(\loc^S_{\frp_i}(\rho))\big) \in \Z/p\Z, \quad \rho \in \cF_{{X_S}}.
\eea
We will sometimes use the notation $ \loc^S_{\frp_i}( \beta_\rho) - x_{\frp_i}(\loc^S_{\frp_i}(\rho))= \inv_{\frp_i} \big( \loc^S_{\frp_i}( \beta_\rho) - x_{\frp_i}(\loc^S_{\frp_i}(\rho))\big)$ for simplicity, when there is no chance of confusion.
We also associate a $\C$-line bundle ${CS}_S^{x_S}$ to  $\mathcal{CS}_S^{x_S}$:
$$
{CS}_S^{x_S}:=\mathcal{CS}_S^{x_S}\times_{\Z/p\Z} \C = \cF_S \times \C.
$$
Then there is an isomorphism $\Theta_S:{CS}_S \to {CS}_S^{x_S}$ as $\C$-line bundles over $\cF_S$.

A benefit of choosing a section $x_S$ to $\mathcal{CS}_S$ is that we can view sections of ${CS}_S^{x_S}$ as a function space, i.e.,
$$
\cH_S^{x_S}=\Gamma(\cF_S, {CS}_S^{x_S})=Map_G(\cF_S, \C).
$$
For any sections $x_S, x_S'$ of $\varpi_S$, there are isomorphisms of $\C$-vector spaces $\Theta_S^{x_S}:\cH_S \xrightarrow{\sim} \cH_S^{x_S}$ and $\Theta_S^{x_S, x_S'}:\cH_S^{x_S} \xrightarrow{\sim} \cH_S^{x_S'}$ such that $\Theta_S^{x_S, x_S'} \circ \Theta_S^{x_S} =\Theta_S^{x_S'}$, which implies that the Hermitian inner product on $\cH_S^{x_S}$ given by 
\begin{eqnarray}\label{HIP}
\langle f,g \rangle:=\sum_{x \in \cF_S}  f(x) \cdot \overline{g(x)}, \quad f,g \in Map_G(\cF_S,\C)=\cH_S^{x_S}
\end{eqnarray}
where $\overline{g(x)}$ is the complex conjugation of $g(x)$, transports to a canonical Hilbert space structure
on $\cH_S$ via $\Theta_S^{x_S}$. Let $||v||^2=\langle v, v \rangle$ for $v \in \cH_S^{x_S}$. 
Using a section $x_S$, \eqref{nvector} can be interpreted as
\bea\label{adwi}
\tilde Z_{{X_S}}^{x_S} (\rho_S)=\frac{1}{p} \sum_{\rho \in \cF_{{X_S}}(\rho_S)} \zeta_p^{\mathcal{CS}_{{X_S}}^{x_S}(\rho)} \in \C, \quad  Z_{{X_S}}^{x_S} =\frac{\tilde Z_{{X_S}}^{x_S} }{||\tilde Z_{{X_S}}^{x_S}||}.
\eea
Then $\tilde Z_{{X_S}}^{x_S} \in \cH_S^{x_S}$ and $\Theta_S^{x_S,x_S'} (\tilde Z_{{X_S}}^{x_S}) = \tilde Z_{{X_S}}^{x_S'}$.
We have a commutative diagram of isomorphisms of Hilbert spaces
\[\begin{tikzcd}
	{\cH_S} && {\cH_{S_1}\otimes \cH_{S_2}} \\
	\\
	{\cH_S^{x_S}} && {\cH_{S_1}^{x_{S_1}}\otimes \cH_{S_2}^{x_{S_2}}}
	\arrow["{\Theta_{S,x_S}^{S_1,S_2}}"', from=3-1, to=3-3]
	\arrow["{\Theta_{S}^{S_1,S_2}}", from=1-1, to=1-3]
	\arrow["{\Theta_S^{x_S}}"', from=1-1, to=3-1]
	\arrow["{\Theta_{S_1}^{x_{S_1}}\otimes \Theta_{S_2}^{x_{S_2}}}", from=1-3, to=3-3]
\end{tikzcd}\]
for any partition $S =S_1\sqcup S_2$ and any section $x_S$ of $\varpi_S:\mathcal{CS}_S \to \cF_S$ together with $x_{S_1}$ and $x_{S_2}$. 
Then we have
$$
\Ent(Z_{X_{S_1,S_2}}) =\Ent((\Theta_{S_1}^{x_{S_1}}\otimes \Theta_{S_2}^{x_{S_2}})(Z_{X_{S_1,S_2}})), \quad 
Z_{X_{S_1,S_2}} = \Theta_S^{S_1,S_2} (Z_{{X_S}}), \
Z_{{X_S}} \in \cH_S
$$
for any choice of $x_{S_1}$ and $x_{S_2}$. 
In other words, the entanglement entropy does not depend on the choices of sections $x_{S_1}$ and $x_{S_2}$. 
Our strategy is to find a suitable section $x_S$ so that 
$\Ent((\Theta_{S_1}^{x_{S_1}}\otimes \Theta_{S_2}^{x_{S_2}})(Z_{X_{S_1,S_2}}))$ can be computed in a simple way.

%


%

\section{Explicit computations}

\subsection{Preliminaries}
Note that the $\bF_p$-vector space $\cF_v$ is equipped with a symplectic paring which is given by the local duality theorem of Tate:
\[ \begin{tikzcd}
\cF_v \arrow{d}{=} &\times& \cF_v \arrow{r}{} \arrow{d}{\simeq} & \arrow{d}{} \bF_p \arrow{d}{\inv_v^{-1}} &  \\%
H^1(\Pi_v, \Z/p\Z) &\times& H^1(\Pi_v, \mu_p) \arrow{r}{} & H^2(\Pi_v,\Z/p\Z) \simeq H^2(\Pi_v, \mu_p). &  
\end{tikzcd}
\]

In the following lemma, $S$ can be any finite set of primes of ${F}$.
\begin{lem} \label{lgr}
Let $\cF_v^{ur}$ be the unramified cohomology subgroup of $\cF_v$.
 
(1)
The image $\loc_S(\cF_{{X_S}})$ is a Lagrangian subspace of $\cF_S=\prod_{v\in S} \cF_v$. 

(2) If $v \in S$ and $v$ does not divide $p$, then $\cF_v^{ur}$ is a Lagrangian subspace of $\cF_v$.
\end{lem}
\begin{proof}
\emph{(1)} follows from the Tate-Poitou exact sequence \cite[Theorem 4.10]{Mil}.

\emph{(2)} is well-known; it follows from the inflation-restriction exact sequence, the local duality, and the Euler-Poincar\'e characteristic formula.
\end{proof}


We recall the definition of the $p$-Selmer group with respect to $S$:
$$
Sel^S({F}, \Z/p\Z) :=\Ker \left(H^1({F}, \Z/p\Z) \to \prod_{v \notin S} H^1({F}_v, \Z/p\Z)/H^1_{ur}({F}_v, \Z/p\Z) \right)
$$
where $H^1_{ur}({F}_v, \Z/p\Z)=H^1(\pi_v,\Z/p\Z)$ is the unramified cohomology subgroup of $H^1({F}_v, \Z/p\Z)$. 
Let $U=\cO_{{F}}^\times$ and  $U_S=\cO_{{F}}[1/S]^\times$. 
\begin{definition}
Define $Cl({X_S})=H^1({X_S}, \mathbb{G}_m)$ to be the class group of ${X_S}$, which is isomorphic to the quotient of the usual ideal class group $Cl(X)$ of ${F}$ by the subgroup generated by the classes of all prime ideals in $S$ (see Proposition 8.3.11. (ii), \cite{NSW}).

\end{definition}

Let $A[p]$ denote the $p$-torsion of an abelian group $A$. Then the Kummer theory gives us the following commutative diagram:
\[\begin{tikzcd}
	{{F}^\times/({F}^\times)^p} && {H^1(F,\mu_p)\simeq H^1({F}, \Z/p\Z)=\Hom(\Gal(\overline{{F}}/{F}), \Z/p\Z)} \\
	\\
	{U_S/U_S^p} && {Sel^S({F},\Z/p\Z)\simeq \cF_{{X_S}}} && {Cl({X_S})[p]} 
	\arrow[hook, from=3-1, to=3-3]
	\arrow[two heads, from=3-3, to=3-5]
	\arrow[hook, from=3-1, to=1-1]
	\arrow[hook, from=3-3, to=1-3]
	\arrow["\simeq", from=1-1, to=1-3]
\end{tikzcd}\]
where the first row is the Kummer isomorphism, an isomorphism ${Sel^S({F},\Z/p\Z)\simeq \cF_{{X_S}}}$ is given in \cite[Lemma5.3]{Rubin} (since $S$ contains all the prime ideals dividing $p$), and we refer to the proof of \cite[(8.7.4)]{NSW} for the exact sequence of the second row.



\subsection{Choices of sections in the case $Cl({X_S})[p]=0$} \label{sub3.2}

We assume that $Cl({X_S})[p]=0$, which implies that the Kummer map
$\kappa_p: U_S/U_S^{p} \simeq \cF_{{X_S}}$ is isomorphic.


%

We construct a section $x_{\frp_i}$ to $\varpi_{\frp_i}:\mathcal{CS}_{\frp_i} \to \cF_{\frp_i}$ using the splitting $s$ given in \eqref{cchoice} and the Kummer isomorphism:
we define 
\begin{eqnarray} \label{localsection}
x_{\frp_i} (\rho_{\frp_i}) = -\rho_{\frp_i} \cup (s \circ \rho_{\frp_i} - \kappa_{p^2} ^{\frp_i}(\tilde u_{\rho_{\frp_i}})), \quad \rho_{\frp_i} \in H^1(\Pi_{\frp_i}, \Z/p\Z)
\end{eqnarray}
where 
$$
\kappa_{p^2} ^{\frp_i}: {F}_{\frp_i}^\times /( {F}_{\frp_i}^\times)^{p^2} \simeq H^1(\Pi_{\frp_i}, \Z/p^2\Z)
$$ 
is the local Kummer isomorphism and $\tilde u_{\rho_{\frp_i}} \in {F}_{\frp_i}^\times /( {F}_{\frp_i}^\times)^{p^2}$ is chosen such that
$$
\kappa_{p^2} ^{\frp_i}( \tilde u_{\rho_{\frp_i}} )(\g) \equiv \rho_{\frp_i}(\g) \mod p, \quad \g \in \Pi_{\frp_i}.
$$ 
Then $d \big( x_{\frp_i} (\rho_{\frp_i}) \big) = c \circ \rho_{\frp_i}$.
Let $U_S$ be the group of global $S$-units of ${F}$.
Then we have the global Kummer isomorphism (due to the assumption $Cl({X_S})[p]=0$)
$$
\kappa_p: U_S/(U_S)^{p} \simeq H^1(\Pi^S, \Z/p\Z)=\Hom(\Pi^S, \Z/p\Z)=:\cF_{{X_S}}
$$

Using the global Kummer isomorphism (using the running assumption $\mu_{p^2} \subset F$)
$$
\kappa_{p^2}: U_S/(U_S)^{p^2} \simeq H^1(\Pi^S, \Z/p^2\Z),
$$
for given $\rho \in H^1(\Pi^S, \Z/p\Z)$ we choose $\tilde u_\rho \in U_S/(U_S)^{p^2}$ such that $\kappa_{p^2} (\tilde u_\rho) (\g) \equiv \rho(\g) \mod p$ for $\g \in \Pi^S$.
Let 
\begin{eqnarray}\label{globalsection}
\beta_\rho= -\rho \cup (s \circ \rho - \kappa_{p^2}(\tilde u_\rho)) \in C^2(\Pi^S, \Z/p\Z), \quad \rho \in H^1(\Pi^S, \Z/p\Z).
\end{eqnarray}
Then it satisfies \eqref{betarho}.


The Hilbert space $\cH_S^{x_S} =Map_G(\cF_S, \C)$ is a finite dimensional $\C$-vector space of dimension $|\cF_S|$. 
Let $\delta_{\rho_S}$ be the delta function of $\rho_S$ on $\cF_S$ and $\{\delta_{\rho_S}: \rho_S \in \cF_S\}$ forms an orthonormal $\bC$-basis of $\cH_S^{x_S}$ with respect to the Hermitian inner product \eqref{HIP}. 
Then
\begin{eqnarray*}
\tilde Z_{{X_S}}^{x_S}&=& \sum_{\rho_S \in \cF_S} \tilde Z_{{X_S}}^{x_S} (\rho_S) \cdot \delta_{\rho_S}
=\sum_{\rho_S \in \cF_S}\frac{1}{p} \sum_{\rho \in \cF_{{X_S}}(\rho_S)} \zeta_p^{\mathcal{CS}_{{X_S}}^{x_S}(\rho)}\cdot \delta_{\rho_S} \in \cH_S^{x_S}  
 \\
\tilde Z_{X_{S_1,S_2}}^{x_S}&=&\Theta_{S,x_S}^{S_1, S_2}\big( \tilde Z_{{X_S}}^{x_S} \big)= \sum_{\rho_S=(\rho_{S_1}, \rho_{S_2}) \in  \cF_S} \tilde Z_{{X_S}}^{x_S} (\rho_S) \cdot  \delta_{\rho_{S_1}} \otimes \delta_{\rho_{S_2}}
\end{eqnarray*}
where 
$
\Theta_{S,x_S}^{S_1, S_2}: \cH_S^{x_S} \xrightarrow{\simeq} \cH_{S_1}^{x_{S_1}} \otimes \cH_{S_2}^{x_{S_2}}
$
is a Hilbert space isomorphism sending $\delta_{\rho_S}$ to $\delta_{\rho_{S_1}} \otimes \delta_{\rho_{S_2}}$ and $\tilde Z_{{X_S}}^{x_S} (\rho_S)$ was given in \eqref{adwi}.

We have a formula for \eqref{pboundary} using $\beta_{\rho}$ and $x_{\frp_i}$: for $\rho \in \cF_{{X_S}}(\rho_S)$
\begin{eqnarray*}
\mathcal{CS}_{{X_S}}^{x_S} (\rho)
&=& \sum_{i=1}^r \big(\loc^S_{\frp_i}(\beta_\rho) - x_{\frp_i} (\loc^S_{\frp_i}(\rho)) \big)\\
&=&  \sum_{i=1}^r  \rho_{\frp_i} \cup \big(\loc^S_{\frp_i}( \kappa_{p^2}(\tilde u_\rho)) - \kappa_{p^2}^{\frp_i}(\tilde u_{\rho_{\frp_i}}) \big).
\end{eqnarray*}

Thus the difference between the global Kummer lifting and the local Kummer lifting (to $\mu_{p^2}$) plays a key role in the computation of $Z_{{X_S}}^{x_S}(\rho_S)$.

\subsection{The entanglement entropy formula in the case $Cl({X_S})[p]=0$ and $|S|=2$} \label{subsec3.3}

In this subsection we assume that $Cl({X_S})[p]=0$ and $|S|=2$.
The formula and its verification under this assumption serve as a good guide to the general case (Theorem \ref{gthm}) both conceptually and technically.

Let $\frp_1$ be a prime ideal of ${F}$ dividing $p$, which is inert, and $\frp_2$ be a prime ideal of ${F}$ not dividing $p$.  Let $S=\{\frp_1, \frp_2\}$.

\begin{theorem}\label{mt}
Assume $Cl({X_S})[p]=0$ and $S=\{ \frp_1, \frp_2\}$. 
The entanglement entropy is given by
$$
\Ent(Z_{X_{\frp_1,\frp_2}})=\biggr(\operatorname{Rank}(\loc^S_{\frp_2})-\operatorname{Nullity}(\loc^S_{\frp_1})\biggl)\log p.
$$
\end{theorem}

Note that the Artin reciprocity map provides an isomorphism $Cl(X)\xrightarrow{\substack{rec \\ \sim}} \Pi^{ab}=\Gal({F}^{ur,ab}/{F})$ 
where ${F}^{ur,ab}$ is the maximal unramified abelian extension of $F$.
Also note that 
\be
\Ker(\loc^S_{\frp_1}) \cap \Ker(\loc^S_{\frp_2})=\Hom(\Pi^{ur,ab}_S, \Z/p\Z)
\ee
 where $\Pi^{ur,ab}_S=\Gal({F}^{ur,ab}_S/{F})$ and
${F}^{ur,ab}_S$ is the maximal unramified abelian extension of ${F}$ in which $S$ splits completely, and thus $\Pi^{ur,ab}_S$ is a quotient of $\Pi^{ab}$.
Moreover, one can check that $Cl({X_S})$ is isomorphic to $\Pi^{ur,ab}_S$ as finite abelian groups. Therefore the assumption $Cl({X_S})[p]=0$ implies that
$\loc_S=(\loc^S_{\frp_1},\loc^S_{\frp_2}): \cF_{{X_S}} \to \cF_S=\cF_{\frp_1} \times \cF_{\frp_2}$ is injective.
Let
\be
s_1=\operatorname{Nullity}(\loc^S_{\frp_1}) \quad \text{ and } \quad
t_2=\operatorname{Rank}(\loc^S_{\frp_2}).
\ee
Then we have $0\leq t_2 \leq 2$ and $t_2-s_1 \geq 0$ (due to the injectivity of $\loc_S$).
Theorem \ref{mt} says that the more degenerate the restriction maps $\loc^S_{\frp_1}$ and $\loc^S_{\frp_2}$ are, the smaller the entanglement entropy is.


\begin{lemma}\label{trivial}
If $Cl({X_S})[p]=0$, $|S|=2$ and we choose $x_S$ as in \eqref{localsection}, then  
\be
\tilde Z_{{X_S}}^{x_S} (\rho_S)=\left\{
  \begin{array}{@{}ll@{}}
    \frac{1}{p}, & \text{if}\ \text{$\cF_{{X_S}}(\rho_S)$ is non-empty} \\
    0, & \text{otherwise}
  \end{array}\right.
\ee
where $\cF_{{X_S}}(\rho_S)$ is defined in \eqref{fromglobal}.
\end{lemma}
\begin{proof}
If $\cF_{{X_S}}(\rho_S)$ is empty, the result is clear. So we concentrate on the case that $\cF_{{X_S}}(\rho_S)$ is non-empty.
For given such $\rho_S$, there is a unique $\rho \in \cF_{X_S}$ such that $\loc_S(\rho)=\rho_S$, because $\loc_S$ is injective. The computation in Subsection \ref{sub3.2} says that
\be
 \tilde Z_{{X_S}}^{x_S} (\rho_S) 
&=&\frac{1}{p} \sum_{\rho \in \cF_{{X_S}}(\rho_S)} \zeta_p^{\sum_{i=1}^2 \rho_{\frp_i} \cup \big(\loc^S_{\frp_i}( \kappa_{p^2}(\tilde u_\rho)) - \kappa_{p^2}^{\frp_i}(\tilde u_{\rho_{\frp_i}}) \big)}\\
&=&\frac{1}{p} \cdot  \zeta_p^{\sum_{i=1}^2 \rho_{\frp_i} \cup \big(\loc^S_{\frp_i}( \kappa_{p^2}(\tilde u_\rho)) - \kappa_{p^2}^{\frp_i}(\tilde u_{\rho_{\frp_i}}) \big)}
\quad (\text{$\loc_S$ is injective and $\loc_S(\rho)=\rho_S$}).\\
\ee
Since the assumption $Cl({X_S})[p]=0$ (which implies that the Kummer map
$\kappa_p: U_S/U_S^{p} \simeq \cF_{{X_S}}$ is isomorphic) enables us to choose both global and local Kummer liftings uniformly, we have
\be
\loc^S_{\frp_i}( \kappa_{p^2}(\tilde u_\rho)) - \kappa_{p^2}^{\frp_i}(\tilde u_{\rho_{\frp_i}})=0,
\ee
which implies that $ \tilde Z_{{X_S}}^{x_S} (\rho_S) =1/p$.
\end{proof}

\begin{proof}[Proof of Theorem \ref{mt}]

According to \cite[Corollary (7.3.9)]{NSW}, we have $\bF_p$-vector space isomorphisms
\be
U_S/U_S^p \simeq \cF_{X_S} &\simeq& \bF_p^\mu, \quad \mu=2+\frac{[{F}:\Q]}{2}, \\
\cF_{\frp_1}&\simeq& \bF_p^{\mu_1}, \quad \mu_1=2+[{F}:\Q], \\
\cF_{\frp_2}&\simeq& \bF_p^{\mu_2}, \quad \mu_2=2.
\ee

Recall that we view 
$Z^{x_S}_{X_{\frp_1,\frp_2}}(\rho_{\frp_1},\rho_{\frp_2})$ as $|\cF_{\frp_1}| \times |\cF_{\frp_2}|$ matrix $A=A_{\rho_{\frp_1}}^{\rho_{\frp_2}}$ using the bases $\{\delta_{\rho_{\frp_i}} : \rho_{\frp_i} \in \cF_{\frp_i} \}, i =1, 2$: the $(\rho_{\frp_1}, \rho_{\frp_2})$-entry of $A$ is $\frac{1}{p} \sum_{\rho \in \cF_{{X_S}}(\rho_S)}\zeta_p^{\mathcal{CS}_{{X_S}}^{x_S} (\rho)}$ where $\rho_S=(\rho_{\frp_1}, \rho_{\frp_2}).$
So $A$ is a $p^{2+n}$ by $p^2$ matrix, where $n=[{F}:\Q]$. Because $2 \geq t_2 > 0$ and $t_2-s_1 \geq 0$ (the injectivity of $\loc_S:\cF_{{X_S}} \to \cF_S$), there are only 5 cases to consider.\footnote{The case $t_2=0$ can not happen due to Lemma \ref{lgr} (1), since $\cF_{X_{\frp_2}}$ injects into $\cF_{{X_S}}$ and the image $\loc_{\frp_2}(\cF_{X_{\frp_2}})$ becomes a Lagrangian of $\cF_{\frp_2} \simeq \bF_p^2$.}
Moreover, we observe that the rank of $A$ is $p^{t_2-s_1}$ and the inner product \eqref{HIP} gives rise to a norm $||A||=\sqrt{\sum_{i,j}|a_{ij}|^2}$ of $A$, where $a_{ij}$ is the $(i,j)$-entry of $A$. 

If, for a given local representation $\rho_S=(\rho_{\frp_1}, \rho_{\frp_2})$, there were no global $\rho$ such that $\loc_S(\rho)=\rho_S$, then the corresponding matrix entry of $A_{\rho_{\frp_1}}^{\rho_{\frp_2}}$ is zero. Since $\loc_S$ is injective, only $p^{2+\frac{n}{2}}$-entries of the $p^{2+n} \times p^{2}$-matrix $A$, corresponding to local representations which factor through $\Pi^S$, are non-zero. By Lemma \ref{trivial}, each such entry is $\frac{1}{p}$.

%

In order to compute the entanglement, we need to express $Z^{x_S}_{X_{\frp_1,\frp_2}} $ in the form (using the Schmidt decomposition based on the singular value decomposition)
$$
\tilde Z^{x_S}_{X_{\frp_1,\frp_2}} = \sum_{\rho_S=(\rho_{\frp_1}, \rho_{\frp_2}) \in  \cF_S} A_{\rho_{\frp_1}}^{\rho_{\frp_2}} \cdot  \delta_{\rho_{\frp_1}} \otimes \delta_{\rho_{\frp_2}}= \sum_{i=1}^{{\operatorname{Rank}(A)}} \lambda_i \cdot e_i \otimes f_i, \quad \lambda_i \in \C
$$
where $\{e_1, \cdots, e_s\} \subset \cH_{\frp_1}^{x_{\frp_1}}$ and $\{f_1, \cdots, f_s\} \subset \cH_{\frp_2}^{x_{\frp_2}}$ are orthonormal sets and $s$ is the rank of $A$.
Let $A^t$ be the transpose of $A$. 
Now we compute the $p^2 \times p^2$ matrix $A^{t} A$ and find the nonzero eigenvalues of $A^{t}A$: the $\lambda_i's$ are given by square roots of non-zero eigenvalues of $A^t A$.
By the singular value decomposition there are unitary (in fact, orthogonal since $A$ has entries in real numbers) matrices $U$ (a $p^{2+n} \times p^{2+n}$ matrix) and $V$ (a $p^2\times p^2$ matrix) such that
$A = U \Sigma V^t $
where $\Sigma$ is a $p^{2+n}\times p^2$ matrix whose $(i,i)$-entry is the singular value $\lambda_i$ (where $ 1 \leq i \leq \operatorname{Rank}(A)$) and other entries are all zeros.
Then $\{e_i\}$ (respectively $\{f_i\}$) consists of the first $s$ column vectors of $U$ (respectively $V$).
Then the entanglement entropy is given by
$$
\Ent(Z_{X_{\frp_1,\frp_2}})=\Ent(Z^{x_S}_{X_{\frp_1,\frp_2}})=\sum_{i=1}^{\operatorname{Rank}(A)} -\frac{|\lambda_i|^2}{||A||^2} \log \frac{|\lambda_i|^2}{||A||^2}.
$$

We compute the matrix $A$ (up to permutation of the basis $\delta_{\rho_{\frp_i}}$), $A^tA$, and the entanglement entropy using the above facts in all possible 5 cases.
Note that the entanglement entropy does not depend on such a choice of a basis and $||A||^2=\frac{1}{p^2} p^\mu=p^{\mu-2}=p^{n/2}$.

\begin{itemize}
\item ($s_1=0, t_2=2$: the rank of $A$ is $p^2$ and $\mu=t_1=s_2+2$)
We compute
{\tiny{
$$
A= \begin{pmatrix}
1/p & 0 &  &  & 0 \\
\vdots & 0 &  &  &  \\
1/p & 0 &  &  &  \\
0 & 1/p & 0 &  & \vdots \\
0 & \vdots & 0 &  &  \\
0 & 1/p & 0 & \cdots &  \\
 & 0 & 1/p & \ddots & 0 \\
\vdots &  &  & 0 & 1/p \\
 &  &  & 0 & \vdots \\
0 &  &  & 0 & 1/p \\
0& & \cdots & 0 & 0 \\
0& & \cdots & 0 & \vdots 
\end{pmatrix}, \quad \quad
\frac{A^tA}{||A||^2}= \begin{pmatrix}
p^{-t_2} & 0 & 0 & 0 \\
0 & \ddots & 0 & 0 \\
0 & 0 & \ddots & 0 \\
0 & 0 & 0 & p^{-t_2} 
\end{pmatrix}.
$$
}}
Then $\frac{A^tA}{||A||^2}$ has eigenvalue $p^{-t_2}$ of multiplicity $p^2$ and so
$$
\Ent(Z_{X_{\frp_1,\frp_2}})=p^2 \cdot (-p^{-t_2} \cdot \log (p^{-t_2}))=2\log p=(t_2-s_1)\log p.
$$

\item ($s_1=1, t_2=2$: the rank of $A$ is $p$ and $\mu=1+t_1=s_2+2$)
We compute
{\tiny{
$$ 
A=\begin{pmatrix}
1/p & \cdots & 1/p &  &  &  & 0 & 0 & 0 \\
\vdots & \ddots & \vdots &  & \cdots &  & 0 & \ddots & 0 \\
1/p & \cdots & 1/p &  &  &  & 0 & 0 & 0 \\
 &  &  &  &  &  &  &  &  \\
 & \vdots &  &  & \ddots &  &  & \vdots &  \\
 &  &  &  &  &  &  &  &  \\
0 & 0 & 0 &  &  &  & 1/p & \cdots & 1/p \\
0 & \ddots & 0 &  & \cdots &  & \vdots & \ddots & \vdots \\
0 & 0 & 0 &  &  &  & 1/p & \cdots & 1/p\\
0 & \vdots & 0 &  &  &  & 0 & 0& 0\\
0 & \cdots & 0 &  &  &  & 0& \cdots & 0
\end{pmatrix},
\quad \quad
\frac{A^tA}{||A||^2}=\begin{pmatrix}
p^{-t_2} & \cdots & p^{-t_2} &  &  &  & 0 & 0 & 0 \\
\vdots& \ddots & \vdots &  & \cdots &  & 0 & \ddots & 0 \\
p^{-t_2} & \cdots & p^{-t_2} &  &  &  & 0 & 0 & 0 \\
 &  &  &  &  &  &  &  &  \\
 & \vdots &  &  & \ddots &  &  & \vdots &  \\
 &  &  &  &  &  &  &  &  \\
0 & 0 & 0 &  &  &  & p^{-t_2} & \cdots & p^{-t_2} \\
0 & \ddots & 0 &  & \cdots &  & \vdots & \ddots & \vdots \\
0 & 0 & 0 &  &  &  & p^{-t_2} & \cdots & p^{-t_2} 
\end{pmatrix}.
$$
}}
Then $\frac{A^tA}{||A||^2}$ has eigenvalue $p^{-t_2+1}$ of multiplicity $p$ and so
$$
\Ent(Z_{X_{\frp_1,\frp_2}})=p\cdot (-p^{-t_2+1} \cdot \log (p^{-t_2+1}))=\log p=(t_2-s_1)\log p.
$$

\item ($s_1=2, t_2=2$: the rank of $A$ is $1$ and $\mu=2+t_1=s_2+2$)
We compute
{\tiny{
$$
A=\begin{pmatrix}
1/p & \cdots & 1/p &  &  &  & 1/p & \cdots & 1/p \\
\vdots & \ddots & \vdots &  & \cdots &  &\vdots  & \ddots & \vdots \\
1/p & \cdots & 1/p &  &  &  & 1/p & \cdots & 1/p \\
 0&  0&  0&\cdots  & 0 & \cdots &  0& 0 & 0 \\
 & \vdots &  &  & \ddots &  &  & \vdots &  \\
 &  &  &  &  &  &  &  &  \\
0 & 0 & 0 &  &  &  & 0 & \cdots &0 \\
0 & \ddots & 0 &  & \cdots &  & \vdots & \ddots & \vdots \\
0 & 0 & 0 &  &  &  & 0& \cdots & 0\\
\end{pmatrix},
\quad \quad
\frac{A^tA}{||A||^2}=\begin{pmatrix}
p^{-t_2}  & p^{-t_2}  & \cdots & p^{-t_2}  & p^{-t_2}  \\
p^{-t_2}  & p^{-t_2}  & \cdots & p^{-t_2}  & p^{-t_2}  \\
\vdots & \vdots & \ddots & \vdots & \vdots \\
p^{-t_2}  & p^{-t_2}  & \cdots & p^{-t_2}  & p^{-t_2}  \\
p^{-t_2}  & p^{-t_2}  & \cdots & p^{-t_2}  & p^{-t_2}  
\end{pmatrix} .
$$
}}
Then $\frac{A^tA}{||A||^2}$ has eigenvalue $p^{-t_2+2}$ of multiplicity $1$ and so
$$
\Ent(Z_{X_{\frp_1,\frp_2}})=-p^{-t_2+2} \cdot \log (p^{-t_2+2})=0=(t_2-s_1)\log p.
$$

\item ($s_1=0, t_2=1$: the rank of $A$ is $p$ and $\mu=t_1=s_2+1$)
We compute 
{\tiny{
$$
A=\begin{pmatrix}
1/p & 0& 0 &  0& 0 &  & 0 & \cdots & 0 \\
\vdots &0 &  & \vdots & \cdots &  &\vdots  & \ddots & \vdots \\
1/p & 0 & 0 &  0&  &  & 0 & \cdots & 0 \\
 0&  1/p&  0&\cdots  & 0 & \cdots &  0& 0 & 0 \\
 & \vdots &  &  &&  &  & \vdots &  \\
 0&  1/p&  &  &  &  &  &  &  \\
0 & 0 & \ddots &  &  &  & 0 & \cdots &0 \\
0 & \ddots & 0 &1/p  & 0 &  & \vdots & \ddots & \vdots \\
0 & 0 & 0 &\vdots  &  &  & 0& \cdots & 0\\
0 & 0 & 0 & 1/p & 0 &  & &  &  \\
0 & \vdots & 0 &   \vdots&  &  & 0& \cdots & 0 \\
0 & 0 & 0 &   0&  0&  & 0& \cdots & 0
\end{pmatrix},
\quad \quad
\frac{A^tA}{||A||^2}= \begin{pmatrix}
p^{-t_2} & 0 & 0 & 0 & \cdots & 0 \\
0 & \ddots & 0 & 0 & \cdots & 0 \\
0 & 0 & p^{-t_2} & 0 & \cdots & 0 \\
0 & 0 & 0 & 0 & \cdots & 0 \\
\vdots & \vdots & \vdots & 0 & \ddots & 0 \\
0 & 0 & 0 & 0 & \cdots & 0 
\end{pmatrix}.
$$
}}
Then $\frac{A^tA}{||A||^2}$ has eigenvalue $p^{-t_2}$ of multiplicity $p$ and so
$$
\Ent(Z_{X_{\frp_1,\frp_2}})=p\cdot(-p^{-t_2} \cdot \log (p^{-t_2}))=\log p=(t_2-s_1)\log p.
$$

\item ($s_1=1, t_2=1$: the rank of $A$ is $1$ and $\mu=1+t_1=s_2+1$)
We compute
{\tiny{
$$
A=\begin{pmatrix}
1/p & \cdots & 1/p &0  &  &  & 0 & \cdots & 0 \\
\vdots & \ddots & \vdots &0  & \cdots &  &\vdots  & \ddots & \vdots \\
1/p & \cdots & 1/p & 0 &  &  & 0 & \cdots & 0 \\
 0&  0&  0&\cdots  & 0 & \cdots &  0& 0 & 0 \\
 & \vdots &  &  & \ddots &  &  & \vdots &  \\
 &  &  &  &  &  &  &  &  \\
0 & 0 & 0 &  &  &  & 0 & \cdots &0 \\
0 & \ddots & 0 &  & \cdots &  & \vdots & \ddots & \vdots \\
0 & 0 & 0 &  &  &  & 0& \cdots & 0\\
\end{pmatrix},
\quad \quad
\frac{A^tA}{||A||^2}= \begin{pmatrix}
p^{-t_2} & \cdots & p^{-t_2} & 0 & 0 & 0 \\
\vdots & \ddots & \vdots & 0 & \cdots & 0 \\
p^{-t_2} & \cdots & p^{-t_2} & 0 & 0 & 0 \\
0 & 0 & 0 & 0 & 0 & 0 \\
0 & \vdots & 0 & 0 & \ddots & 0 \\
0 & 0 & 0 & 0 & 0 & 0 
\end{pmatrix}.
$$
}}
Then $\frac{A^tA}{||A||^2}$ has eigenvalue $p^{-t_2+1}$ of multiplicity $1$ and so
$$
\Ent(Z_{X_{\frp_1,\frp_2}})=-p^{-t_2+1} \cdot \log (p^{-t_2+1})=0=(t_2-s_1)\log p.
$$

\end{itemize}
\end{proof}

\subsection{A generalized formula without $Cl({X_S})[p]=0$} \label{subsec4.1}
Here we drop the assumption $Cl({X_S})[p]=0$ and $|S|=2$.
We use the glueing formula of ACST to derive the generalized formula in Theorem \ref{gthm}.
Choose any finite set $S_3$ of primes disjoint from $S$ and write
\be
S':=S \sqcup S_3
\ee
For $x_S \in \Gamma(\cF_{S}, \mathcal{CS}_S)$ and $x_{S_3} \in \Gamma(\cF_{S_3}, \mathcal{CS}_{S_3})$, we define $x_{S'} \in \Gamma(\cF_{S'}, \mathcal{CS}_{S'})$ by
\bea \label{sectionsum}
x_{S'}(\rho_S, \rho_{S_3}) = [x_S(\rho_S) + x_{S_3}(\rho_{S_3})].
\eea
The glueing formula (\cite[Theorem 5.2.1]{HKM}) gives that for $\rho\in \cF_{{X_S}}$
\begin{eqnarray}\label{glue}
\mathcal{CS}_{{X_S}}^{x_S} (\rho) = \mathcal{CS}_{X_{S'}}^{x_{S'}}( \rho\circ \eta^{S',S}) - \mathcal{CS}_{ V_{S_3}}^{x_{S_3}} ((\rho \circ u_\frp)_{\frp \in S_3})
\end{eqnarray}
where $\eta^{S',S}: \Pi^{S'} \to \Pi^S$ is the natural quotient map, $u_\frp: \tilde \Pi_\frp:=\Gal({F}_\frp^{ur}/{F}_\frp) \to \Pi^S$ is the embedding for $\frp \in S_3$, and for $\tilde \rho_\frp \in \Hom(\tilde \Pi_\frp, \mu_p)$ 
$$
\mathcal{CS}_{V_{S_3}}^{x_{S_3}}((\tilde \rho_\frp)_{\frp \in S_3}) :=
 \sum_{\frp \in S_3}\inv_\frp\left(\tilde \beta_{\frp} -x_\frp (\tilde \rho_\frp) \right)\in \Z/p\Z
$$
with $d(\tilde \beta_\frp) =c \circ \tilde \rho_{\frp}$ (using the unramified trivialisation $H^3(\tilde \Pi_\frp, \Z/p\Z)=0$).
Define
\begin{eqnarray}\label{unr}
Z_{V_{S_3}^\ast}^{x_{S_3}} (\rho_{S_3}):= \sum_{\cF_{V_{S_3}}(\rho_{S_3})\ni(\tilde \rho_\frp)_{\frp \in S_3}} \zeta_p^{\mathcal{CS}_{V_{S_3}^\ast}^{x_{S_3}}((\tilde \rho_\frp)_{\frp \in S_3})}
:= \sum_{\cF_{V_{S_3}}(\rho_{S_3})\ni (\tilde \rho_\frp)_{\frp \in S_3}} \zeta_p^{-\mathcal{CS}_{V_{S_3}}^{x_{S_3}}((\tilde \rho_\frp)_{\frp \in S_3})}
\end{eqnarray}
and $\cF_{V_{S_3}}(\rho_{S_3})=\{ \tilde \rho=(\tilde \rho_\frp)_{\frp \in S_3} \in  \prod_{\frp \in S_3} \Hom(\tilde \Pi_\frp, \mu_p) : \tilde \rho_\frp \circ q_\frp = \rho_\frp, \quad \frp \in S_3\}$ (here $q_\frp:\Pi_\frp \to \tilde \Pi_\frp$ is the natural quotient map).
Then $Z_{V_{S_3}^\ast}^{x_{S_3}} \in \cH_{S_3^\ast}^{x_{S_3}}$ and $Z_{V_{S_3}^\ast}^{x_{S_3}}(\rho_{S_3})$ are independent of choices of $\tilde \beta_\frp$ (see \cite[Theorem 5.1.4]{HKM}). 

\begin{lemma}\label{inner}\cite[Theorem 5.2.4]{HKM}
For any section $x_{S'}$ satisfying \eqref{sectionsum}, we have that
$$
Z_{{X_S}}^{x_S}(\rho_S) =\sum_{\rho_{S_3} \in \cF_{S_3}^{ur}} Z_{X_{S'}}^{x_{S'}} (\rho_{S}, \rho_{S_3}) \cdot Z_{V_{S_3}^\ast}^{x_{S_3}} (\rho_{S_3}), \quad (\rho_{S},\rho_{S_3}) = \rho_{S'}.
$$
\end{lemma}
\begin{proof}
The equality follows from the glueing formula \eqref{glue}, the definitions of $Z_{{X_S}}^{x_S}, Z_{X_{S'}}^{x_{S'}}$ and $Z_{V_{S_3}^\ast}^{x_{S_3}}$ based on the following bijection\footnote{The surjectivity of $\eta^{S',S}$ implies the injectivity of the map. For the surjectivity of the map, consider $(\tau, (\tilde \tau_\frp)_{\frp \in S_3}) \in \cF_{X_{S'}}(\rho_S,\rho_{S_3}) \times \cF_{V_{S_3}}(\rho_{S_3})$. Then there exists $\rho \in \cF_{{X_S}}$ such that $\tau = \rho \circ \eta^{S',S}$, since $\loc^{S'}_{\frp}(\tau)$ is unramified for $\frp \in S_3$, and $\rho \circ u_\frp =\tilde \tau_\frp$ for $\frp \in S_3$ due to the surjectivity of $q_\frp$. }
\begin{eqnarray}\label{bij}
\cF_{{X_S}}(\rho_S) \simeq 
\bigsqcup_{\rho_{S_3} \in \cF_{S_3}}
 \left(\cF_{X_{S'}}(\rho_S,\rho_{S_3}) \times \cF_{V_{S_3}}(\rho_{S_3})\right), 
\quad \rho \mapsto (\rho \circ \eta^{S',S}, (\rho\circ u_\frp)_{\frp \in S_3})
\end{eqnarray}
where we recall $\cF_{X_{S'}}(\rho_{S'})=\{ \rho \in \cF_{X_{S'}} : \loc_{S'}(\rho) = \rho_{S'}\}$.
In fact, there is a bijection
$$
\bigsqcup_{\rho_{S_3} \in \cF_{S_3}}
 \left(\cF_{X_{S'}}(\rho_S,\rho_{S_3}) \times \cF_{V_{S_3}}(\rho_{S_3})\right) \simeq 
 \bigsqcup_{\rho_{S_3} \in \cF_{S_3}^{ur}}
 \left(\cF_{X_{S'}}(\rho_S,\rho_{S_3}) \times \cF_{V_{S_3}}(\rho_{S_3})\right)
 $$
  where $\cF_{S_3}^{ur}$ consists of unramified representations of $\cF_{S_3}$.
\end{proof}

\begin{lemma}\label{triviality}
If we choose any finite set $S_3$ of primes which is disjoint from $S$ such that $Cl(X_{S \sqcup S_3})[p]=0$ (such $S_3$ always exists), then
there exists a local section $x_{S\sqcup S_3}$ to the $\Z/p\Z$-torsor map $\varpi_{S\sqcup S_3}:\mathcal{CS}_{S\sqcup S_3} \to \cF_{S \sqcup S_3}$ such that
\be
\tilde Z_{{X_S}}^{x_S} (\rho_S)=\left\{
  \begin{array}{@{}ll@{}}
  \frac{1}{p} \sum_{\rho_{S_3} \in \cF_{S_3}^{ur}} Z_{V_{S_3}^\ast}^{x_{S_3}} (\rho_{S_3}), & \text{if}\  \rho_S \in \loc_{S}(\cF_{X_{S}}) \\
    0, & \text{otherwise}
  \end{array}\right.
\ee
where
$
Z_{V_{S_3}^\ast}^{x_{S_3}} (\rho_{S_3})$ is given in \eqref{unr}. In particular, $\tilde Z_{{X_S}}^{x_S}(\rho_S)$ obtains the same value regardless of $\rho_S$, once we fix $S_3$.
\end{lemma}
\begin{proof}
By Lemma \ref{trivial} (note that the same proof works for an arbitrary partition $S=S_1\sqcup S_2$, though Lemma \ref{trivial} was stated in the case $\#S_1=\#S_2=1$), there is a section $x_{S'}$ ($S'=S \sqcup S_3$) such that
$$
\tilde Z_{X_{S'}}^{x_{S'}}(\rho_{S'}) = \tilde Z_{X_{S'}}^{x_{S'}}(\rho_{S},\rho_{S_3}) =\frac{1}{p}, \quad \rho_{S'} \in \loc_{S'}(\cF_{X_{S'}})
$$
because of $Cl(X_{S'})[p] =Cl(X_{S \sqcup S_3})[p]=0$.
Observe that $\rho_S \in \loc_{S}(\cF_{X_{S}})$ implies $\rho_{S'} \in \loc_{S'}(\cF_{X_{S'}})$, since $\rho_{S_3} \in \cF_{S_3}^{ur}$.
Thus Lemma \ref{inner} implies that
$$
\tilde Z_{{X_S}}^{x_S}(\rho_S) =  \sum_{\rho_{S_3} \in \cF_{S_3}^{ur}} p^{-1}\cdot Z_{V_{S_3}^\ast}^{x_{S_3}} (\rho_{S_3}),\quad \rho_{S} \in \loc_S(\cF_{X_S}).
$$
If $\rho_{S} \notin \loc_{S}(\cF_{X_{S}})$, the result clearly holds by definition \eqref{adwi}.
\end{proof}
Now we prove the generalized formula (Theorem \ref{gthm}). Choose subsets $S_1,S_2$ of $S$ such that $S=S_1 \sqcup S_2$ ($S_1 \cap S_2 =\phi$) and neither $S_1$ nor $S_2$ is empty.
\begin{proof}[Proof of Theorem \ref{gthm}]
We have the following commutative diagram
\[\begin{tikzcd}
	&&&& {\cF_{S_1}} \\
	{\cF_X^S} && {\cF_{X_S}} &&& {} & {\cF_S=\cF_{S_1}\times\cF_{S_2}} \\
	&&& {} & {\cF_{S_2}}
	\arrow["{\loc^S_{S_1}}", from=2-3, to=1-5]
	\arrow["{\loc^S_{S_2}}"', from=2-3, to=3-5]
	\arrow[hook, from=2-1, to=2-3]
	\arrow["{\loc_S=\loc_S^S}", from=2-3, to=2-7]
	\arrow["{pr_{_{S_1}}}"', from=2-7, to=1-5]
	\arrow["{pr_{_{S_2}}}", from=2-7, to=3-5]
\end{tikzcd}\]
where $\cF_X^S:=\Ker(\loc_S)\simeq \Hom(\Pi^{ur,ab}_S,\Z/p\Z)$ (recall that $\Pi^{ur,ab}_S$ is the maximal unramified abelian extension of ${F}$ such that all the primes above $S$ split), $\loc^S_{S_i}$ is the map induced from the embedding $\Pi_\frp \to \Pi^S$ ($\frp \in S_i$), and $pr_{S_i}$ is the projection map to $\cF_{S_i}$.

Let $\nu$ be the $\bF_p$-dimension of $\cF_X^S$.
Let $s_i$ be the $\bF_p$-dimension of the kernel of $\loc^S_{S_i}$ for each $i$.
Let $t_i$ be the $\bF_p$-dimension of the image of $\loc^S_{S_i}$ for each $i$.
Define $\mu_i :=\dim_{\bF_p}\cF_{S_i}$ and $\mu:=\frac{\mu_1+\mu_2}{2}$. Then $\dim_{\bF_p}(\cF_{X_S})=\mu+\nu$ by Lemma \ref{lgr}.
Then we have
\be
t_2-s_1+\nu=\left(\dim_{\bF_p}(\cF_{X_S})-\dim_{\bF_p} (\Ker(\loc^S_{S_1}) + \Ker(\loc^S_{S_2}))\right),
\ee
because $t_2-s_1+\nu = \mu-s_1-s_2+2 \nu$ (due to $\mu+\nu=s_2+t_2$) and so $t_2-s_1+\nu$ is the $\bF_p$-dimension $\mu+\nu$ of $\cF_{X_S}$ minus the $\bF_p$-dimension $s_1+s_2-\nu$ of the space $\Ker(\loc^S_{S_1}) + \Ker(\loc^S_{S_2})$.

We choose a section $x_{S'}$ as in Lemma \ref{triviality}.
By using the same technique as in the proof of Theorem \ref{mt}, we compute the matrix $A$ associated to $\tilde Z_{X_{S_1,S_2}}^{x_S}$ (possibly permuting the basis elements $\delta_{\rho_S}$ of $\cF_S$), the matrix $A^tA$, and the entanglement entropy 

The matrix $A$ has the following 3 properties:
\begin{itemize}
\item there are $p^{t_2-s_1+\nu}$ linearly independent non-zero columns;
\item if a row vector is non-zero, there are only $p^{s_1-\nu}$ non-zero entries in that row vector;
\item there are only $p^\mu$ non-zero entries.
\end{itemize}

In fact, the $p^{\mu_1}\times p^{\mu_2}$ matrix $A$ is a block diagonal type matrix whose each block is a $p^{\mu-t_2}\times p^{s_1-\nu}$ matrix with each entry\footnote{We use the notation $\mu_{_{S_3}}$, since it only depends on the unramified representations in $\cF_{S_3}$.}
$$
\mu_{_{S_3}}:= \tilde Z_{X_S}^{x_S}(\rho_S) =  p^{-1} \cdot \sum_{\rho_{S_3} \in \cF_{S_3}} Z_{V_{S_3}^\ast}^{x_{S_3}} (\rho_{S_3})
$$
 and the number of the non-zero blocks is $p^{t_2-s_1+\nu}$:
\[
  \setlength{\arraycolsep}{0pt}
 A= \begin{pmatrix}
    \fbox{$B_1$} & 0 & \cdots &&0 &\\
    0 & \fbox{$B_2$} & 0 & \vdots &&\\
    0 & 0 & \ddots & 0& \vdots &\\
    \vdots & \vdots & 0  &\fbox{$B_{p^{t_2-s_1+\nu}}$} & &0\\
    0 & 0 &\cdots &0& \cdots& 0\\
  \end{pmatrix}, \quad
  B_i = 
  \begin{pmatrix}
  \mu_{_{S_3}}& \cdots & \mu_{_{S_3}}\\
  \vdots&\ddots&\vdots\\
   \mu_{_{S_3}}&\cdots &  \mu_{_{S_3}}\\
  \end{pmatrix}
\]
where each $B_i$ is a $p^{\mu-t_2}\times p^{s_1-\nu}$ matrix.
Then the $p^{\mu_2}\times p^{\mu_2}$ matrix $A^t A/||A||^2$ is again a block diagonal type matrix whose each block is a $p^{s_1-\nu} \times p^{s_1-\nu}$ matrix with all the entries 
$$
p^{-t_2}=\frac{\mu_{_{S_3}}^2 \cdot p^{\mu-t_2}}{\mu_{_{S_3}}^2 \cdot p^{\mu}}
$$
 and the number of non-zero blocks is $p^{t_2-s_1+\nu}$ (note that $||A||^2=\mu_{_{S_3}}^2 \cdot p^\mu$, since each entry of $A$ is $\mu_{_{S_3}}$):
\[
  \setlength{\arraycolsep}{0pt}
 \frac{A^t A}{||A||^2}= \begin{pmatrix}
    \fbox{$C_1$} & 0 & \cdots &&0 &\\
    0 & \fbox{$C_2$} & 0 & \vdots &&\\
    0 & 0 & \ddots & 0& \vdots &\\
    \vdots & \vdots & 0  &\fbox{$C_{p^{t_2-s_1+\nu}}$} & &0\\
    0 & 0 & && \cdots& 0\\
  \end{pmatrix}, \quad
  C_i = 
  \begin{pmatrix}
  p^{-t_2}& \cdots & p^{-t_2}\\
  \vdots&\ddots&\vdots\\
  p^{-t_2}&\cdots &p^{-t_2}\\
  \end{pmatrix}
\]
where each $C_i$ is a $p^{s_1-\nu}\times p^{s_1-\nu}$ matrix.
Therefore $A^t A/||A||^2$ has eigenvalue $p^{-t_2 +s_1-\nu}$ with multiplicity $p^{t_2-s_1+\nu}$
and so we conclude that
$$
\Ent(Z_{X_{S_1,S_2}})
=(-p^{-t_2 +s_1-\nu} \log p^{-t_2 +s_1-\nu}) \cdot p^{t_2-s_1+\nu} = (t_2-s_1+\nu) \log p.
$$
\end{proof}


\begin{thebibliography}{99}


%


\bibitem{BFLP} 
Balasubramanian, Vijay; Fliss, Jackson R.; Leigh, Robert G.; Parrikar, Onkar: Multi-boundary entanglement in Chern-Simons theory and link invariants. J. High Energy Phys. 2017, no. 4, 061, front matter + 33 pp.
\bibitem{bell}Bell, J.S.: On the Einstein-Podolsky-Rosen paradox. Physics 1, 195 (1964) 

\bibitem{BSV} David Ben-Zvi; Yiannis Sakellaridis; Akshay Venkatesh: Relative Langlands Duality. Preprint (2023), https://www.math.ias.edu/~akshay/research/BZSVpaperV1.pdf

\bibitem{CCKKPY}
Magnus Carlson; Hee-Joong Chung; Dohyeong Kim; Minhyong Kim; Jeehoon Park; Hwajong Yoo: Path Integrals and p-adic L-functions. Preprint (2023), arXiv:2207.03732



\bibitem{CK} M. Carlson; M. Kim: A note on abelian arithmetic BF-theory, Bulletin of the London Mathematical Society. Published online, April, 2022. https://londmathsoc.onlinelibrary.wiley.com/doi/full/10.1112/blms.12629?af=R

%
\bibitem{CKKPPY19}
Hee-Joong Chung; Dohyeong Kim; Minhyong Kim; George Pappas; Jeehoon Park; Hwajong Yoo:
{\it Abelian arithmetic Chern-Simons theory and arithmetic linking numbers,} International Mathematics Research Notices, Volume 2019, Issue 18, September 2019, Pages 5674--5702.
%
%
\bibitem{CKKPY20}
Hee-Joong Chung; Dohyeong Kim; Minhyong Kim; Jeehoon Park; Hwajong Yoo:
{\it Arithmetic Chern-Simons theory II}  pages 81-128, $p$-adic Hodge theory (editors: Bhatt, Bhargav; Olsson, Martin 2020), Simons Symposia, ISBN 978-3-030-43844-9.

\bibitem{FQ} Freed, Daniel S.; Quinn, Frank: Chern-Simons theory with finite gauge group. Comm. Math. Phys. 156 (1993), no. 3, 435?472.


\bibitem{HKM} 
H. Hirano, J. Kim, and M. Morishita: On arithmetic Dijkgraaf-Witten theory.  arXiv:2106.02308v2 


\bibitem{kim}
M. Kim: Arithmetic Chern-Simons theory I. Galois covers, Grothendieck-Teichm\"{u}ller Theory and Dessins d'Enfants, 155-180, Springer Proc. Math. Stat., 330, Springer, Cham, [2020]


\bibitem{NSW}
Neukirch, J.; Schmidt, A.; Wingberg, K.: Cohomology of number fields. Second edition. Grundlehren der mathematischen Wissenschaften [Fundamental Principles of Mathematical Sciences], 323. Springer-Verlag, Berlin, 2008. xvi+825 pp.

\bibitem{Mil}
Milne, J.: Arithmetic duality theorems. Second edition. BookSurge, LLC, Charleston, SC, 2006. viii+339 pp.
%

\bibitem{Rubin}
Rubin, K.: Euler systems. Annals of Mathematics Studies, 147. Hermann Weyl Lectures. The Institute for Advanced Study. Princeton University Press, Princeton, NJ, 2000. xii+227

%




\end{thebibliography}
\end{document}